\documentclass[12pt]{article}
\usepackage{amsfonts}
\usepackage{amsmath,amssymb,amsthm}
\usepackage[usenames]{color}
\usepackage{mathrsfs}
\usepackage{syntonly}
\usepackage[top=1in,bottom=1.4in,left=1in,right=1in]{geometry}
\usepackage[all,cmtip]{xy}
\usepackage{bbm}
\usepackage{mathrsfs}
\usepackage{dsfont}
\usepackage{enumitem}
\usepackage{fouridx}
\usepackage[stable]{footmisc}
\usepackage{array}
\usepackage{caption}
\usepackage{latexsym}
\usepackage{extarrows}
\usepackage{amsfonts}
\usepackage{amsmath}
\usepackage{amsthm}
\usepackage{amssymb}
\usepackage{latexsym}
\usepackage{indentfirst}
\usepackage[T1]{fontenc}
\usepackage{lipsum}
\usepackage{tikz}
\usepackage{tikz-cd}
\usepackage{pgfplots}
\pgfplotsset{compat=1.15}
\usetikzlibrary{calc,arrows, automata, positioning, mindmap,cd,babel,shadows,trees,patterns}
\allowdisplaybreaks

\newtheorem{thm}{Theorem}[section] 
\newtheorem{proposition}[thm]{Proposition}
\newtheorem{corollary}[thm]{Corollary}
\newtheorem{lemma}[thm]{Lemma}
\newtheorem{definition}[thm]{Definition}
\newtheorem{remark}[thm]{Remark}
\numberwithin{equation}{section}

\begin{document}


\title{\Large\bf  Simple weight modules for Yangian $\operatorname{Y}(\mathfrak{sl}_{2})$}
\author{{Yikun Zhou$^{1}$, Yilan Tan$^{1,*}$, Limeng Xia$^{1,2}$}}

\date{} 

\maketitle
\begin{abstract}
  Let $\mathfrak{g}$ be a finite-dimensional simple Lie algebra over $\mathbb{C}$.  A $\operatorname{Y}(\mathfrak{g})$-module is said to be weight if it is a weight $\mathfrak{g}$-module. We give a complete classification of simple weight modules for $\operatorname{Y}(\mathfrak{sl}_2)$ which admits a one-dimensional weight space. We prove that there are four classes of such modules: finite, highest weight, lowest weight and dense modules. Different from the classical $\mathfrak{sl}_{2}$-representation theory, we show that there exist a class of $\operatorname{Y}(\mathfrak{sl}_{2})$ irreducible modules which have uniformly 2-dimensional weight spaces. 
\end{abstract}

\vspace{10 mm}


\vspace{20 mm}
\noindent Author Addresses:

1. School of Mathematical Science, Jiangsu University, Zhenjiang, Jiangsu, 212013, China.

2. Institute of Applied System Analysis, Jiangsu University, Zhenjiang, Jiangsu, 212013, China.

\vspace{5mm}
\noindent Corresponding Author(*) 

\noindent Email: tanyanlan@ujs.edu.cn\vspace{0.5cm}

\noindent The third author of this paper thanks the support of the National Natural Science Foundation of China (Grants No. 11871249, 12171155)




\newpage

\section{Introduction}
Let $\mathfrak{g}$ be a finite-dimensional simple Lie algebra over $\mathbb{C}$. The Yangian $\operatorname{Y}(\mathfrak{g})$ and the quantum affine algebra $\operatorname{U_{q}}({\hat{\mathfrak{g}}})$ constitute two remarkable families of the quantum groups of affine type. The Yangian $\operatorname{Y}(\mathfrak{g})$ is a unital associative algebra which is a Hopf algebra deformation of universal enveloping algebra of the current algebra $\mathfrak{g}[t]$.  In the late 1970s and early 1980s, The Yangian used for the first time in the work of Ludvig Faddeev and his school  concerning the quantum inverse scattering method in statistical mechanics. Later, it appears as a symmetric group in different physics models.  The representation of the Yangian could be used to construct the rational solutions of the quantum Yang-Baxter equation\cite{ChPr4}. As stated in \cite{NT}, ``the physical data such as mass formula, fusion angle, and the spins of integrals of motion can be extracted from the Yangian highest weight representations.'' 

The finite-dimensional simple modules and highest weight modules of Yangians were studied in the last forty years, see \cite{VMFY, ChPr4, ChPr3, KhNaPa, Mo2, TaGu} and the references in. The representation theory of $\operatorname{Y}(\mathfrak{sl}_{2})$ is of paramount importance to  understand the structure of simple $\operatorname{Y}(\mathfrak{g})$-modules for general $\mathfrak{g}$. It was showed by Chari and Pressley in \cite{ChPr3} that every finite-dimensional simple module for $\operatorname{Y}(\mathfrak{sl}_{2})$ is a tensor product of modules which are simple under $\mathfrak{sl}_{2}$. Then the irreducible condition for $\operatorname{Y}(\mathfrak{g})$ follows. Another example is the local Weyl modules for $\operatorname{Y}(\mathfrak{g})$, see \cite{TaGu}. A Local Weyl module $W(\pi)$ is a highest weight object with a nice property: every finite-dimensional hightest weight module associated to $\pi$ is a quotient of $W(\pi)$. The local weyl modules for $\operatorname{Y}(\mathfrak{sl}_{2})$ is isomorphic to an ordered tensor product of fundamental representations of $\operatorname{Y}(\mathfrak{sl}_{2})$. It is crucial to obtain the structure of the local Weyl module $W(\pi)$ for general $\operatorname{Y}(\mathfrak{g})$.     

The purpose of this paper is to construct and classify a new class of simple modules of the Yangian $\operatorname{Y}(\mathfrak{sl}_{2})$. Analogous to the weight modules of current algebras \cite{Lau}, a $\operatorname{Y}(\mathfrak{g})$-module is called weight if it is a  weight $\mathfrak{g}$-module. Both the finite-dimensional simple modules and highest weight modules are weight modules, which have something in common: there exists a one-dimensional weight space. In this paper we study the simple weight module for $\operatorname{Y}(\mathfrak{sl}_{2})$ which admits a one-dimensional weight space. Due to the existence of the evaluation homomorphism for $\operatorname{Y}(\mathfrak{sl}_{n})$, every irreducible $\mathfrak{sl}_{n}$-module is a simple $\operatorname{Y}(\mathfrak{sl}_{n})$-module. The  irreducible dense modules of $\mathfrak{sl}_{2}$  come in as another class of  weight module for $\operatorname{Y}(\mathfrak{sl}_{2})$, which are still called dense in this paper. One of our contributions is to give the structure of dense modules, which are parametrized by three parameters $\mu,\tau,b_{\mu}$, see Section 3 for detail. Our main result is 
\begin{thm}\label{thm1} 
  Let $V$ be a simple weight module for $\operatorname{Y}(\mathfrak{sl}_{2})$ which admits a one-dimensional weight space. Then $V$ is isomorphic to one of the following modules:
  \begin{enumerate}[label=(\arabic*)]
    \item A finite-dimensional simple module.
    \item An infinite-dimensional simple highest weight module.
    \item An infinite-dimensional simple lowest weight module.
    \item A simple dense module $V(\mu,\tau,b_{\mu})$.
  \end{enumerate}
  \end{thm}
It is well known to the Lie theorists that every weight space of any simple weight $ \mathfrak{sl}_{2}$-module is one-dimensional. However, the condition in Theorem \ref{thm1}  that $V$ admits a one-dimensional weight space  is essential. In Section 5 we construct a class of simple weight modules whose every weight space is uniformly two dimensional. 

In this paper, denote the set of natural numbers, the set of positive integers and the set of complex numbers by $\mathbb{N}$, $\mathbb{Z}_{>0}$ and $\mathbb{C}$, respectively. 

\section{Preliminary}\label{sec:1}
In this section, we recall the definition of the Yangian $\operatorname{Y}(\mathfrak{sl}_{2})$ and some results concerning its finite-dimensional representations. For the general definition of the Yangian $\operatorname{Y}(\mathfrak{g})$, we refer the articles \cite{VMFY, ChPr4}.

\begin{definition}
 The Yangian $\operatorname{Y}(\mathfrak{sl}_2)$ 
 is an associative algebra with generators $X_{k}^{\pm}$, $H_{k}$, $k \in \mathbb{N},$
 and the following defining relations:
 \begin{equation*}
   \left[H_{k}, H_{l}\right]=0, \quad\left[H_{0}, X_{k}^{\pm}\right]=\pm 2 X_{k}^{\pm}, \quad\left[X_{k}^{+}, X_{l}^{-}\right]=H_{k+l},
  \end{equation*}
 \begin{equation}\label{e11}
   \left[H_{k+1}, X_{l}^{\pm}\right]-\left[H_{k}, X_{l+1}^{\pm}\right]=\pm\left(H_{k} X_{l}^{\pm}+X_{l}^{\pm} H_{k}\right) ,
  \end{equation}	
 \begin{equation}\label{e12}
 \left[X_{k+1}^{\pm}, X_{l}^{\pm}\right]-\left[X_{k}^{\pm}, X_{l+1}^{\pm}\right]=\pm\left(X_{k}^{\pm} X_{l}^{\pm}+X_{l}^{\pm} X_{k}^{\pm}\right) .
  \end{equation}
\end{definition}

Yangian $\operatorname{Y}(\mathfrak{sl}_{2})$ is a hopf algebra. However, there is no explicit coproduct formula for the realization used in this paper. Partial information were obtained in the paper \cite{ChPr5}, which is enough for ours purpose.

\begin{lemma}[\cite{ChPr5}] 
The coproduct $\Delta$ of $\operatorname{Y}(\mathfrak{sl}_{2})$ satisfies:
\begin{enumerate}[label=(\arabic*)]
 \item $\Delta\left(H_{0}\right)=H_{0} \otimes 1+1 \otimes H_{0}$,
 \item $\Delta\left(H_{1}\right)=H_{1} \otimes 1+H_{0} \otimes H_{0}+1 \otimes H_{1}-2 X_{0}^{-} \otimes X_{0}^{+} $,
 \item $\Delta\left(X_{0}^{+}\right)=X_{0}^{+} \otimes 1+1 \otimes X_{0}^{+}$,
 \item $\Delta\left(X_{1}^{+}\right)=X_{1}^{+} \otimes 1+1 \otimes X_{1}^{+}+H_{0} \otimes X_{0}^{+}$,
 \item $\Delta\left(X_{0}^{-}\right)=X_{0}^{-} \otimes 1+1 \otimes X_{0}^{-}$,
 \item $\Delta\left(X_{1}^{-}\right)=X_{1}^{-} \otimes 1+1 \otimes X_{1}^{-}+X_{0}^{-} \otimes H_{0}.$
\end{enumerate} 
\end{lemma}

\begin{corollary}
 The generators $X_{0}^{\pm}$ and $H_{0}$ generate a subalgebra which is isomorphic to  $\mathfrak{sl}_{2}$.
\end{corollary}

\begin{lemma}[\cite{Guay,Leve}] \label{lem5}
$\operatorname{Y}(\mathfrak{sl}_{2})$ is generated by $X_{0}^{\pm}$, $H_{0}$ and $H_{1}$.
\end{lemma}
\begin{proof} 
It follows from the defining relations of $\operatorname{Y}(\mathfrak{sl}_{2})$ that 
$$[H_0^2, X_k^{\pm}]=H_0[H_0,X_k^{\pm}]+[H_0,X_k^{\pm}]H_0=\pm 2(H_0X_k^{\pm}+X_k^{\pm}H_0).$$
Then we have 
\begin{eqnarray}\label{e13}   
 \left[H_{1}, X_{l}^{\pm}\right]-\left[H_{0}, X_{l+1}^{\pm}\right]	&=& \pm\left(H_{0} X_{l}^{\pm}+X_{l}^{\pm} H_{0}\right)\nonumber\\
   \left[H_{1}, X_{l}^{\pm}\right]\mp 2X_{l+1}^{\pm}&=& \frac{1}{2}[H_0^2, X_l^{\pm}]\nonumber\\
   \mp 2X_{l+1}^{\pm}&=& -\left(\left[H_{1}, X_{l}^{\pm}\right]- \frac{1}{2}[H_0^2, X_l^{\pm}]\right)\nonumber\\
   X_{l+1}^{\pm}&=& \pm \frac{1}{2}\left[H_{1}-\frac{1}{2}H_0^2, X_{l}^{\pm}\right].
\end{eqnarray}
Note that $\left[X_{r}^{+}, X_{0}^{-}\right]=H_{r}$. Thus $\operatorname{Y}(\mathfrak{sl}_2)$ is generated by $X_0^{\pm}$, $H_{0}$ and $H_1$.
\end{proof}

In next we introduce the representation theory of $\operatorname{Y}(\mathfrak{sl}_{2})$.

\begin{definition}
 A module $V(\mu(u))$ of $\operatorname{Y}(\mathfrak{sl}_{2})$ is said to be highest weight if there exists a vector $v^{+}$ such that 
 \begin{center}
   $V(\mu(u))=\operatorname{Y}(\mathfrak{sl}_{2}).v^{+}$, $X_{k}^{+}.v^{+}= 0$ and $H_{k}.v^{+} =\mu_{k} v^{+}$,
 \end{center}
 where $\mu(u)=1+\mu_{0}u^{-1}+\mu_{1}u^{-2}+\ldots$ is a formal series in $u^{-1}$.
 \end{definition}

The Verma module $M(\mu(u))$ of $\operatorname{Y}(\mathfrak{sl}_{2})$ is defined to be the quotient of $\operatorname{Y}(\mathfrak{sl}_{2})$ by the left ideal generated by generators $X_{k}^{+}$ and the elements $H_{k}-\mu_{k}1$. $\operatorname{Y}(\mathfrak{sl}_{2})$ acts on $M(\mu(u))$ by left multiplication. A highest weight vector of $M(\mu(u))$ is $1_{\mu(u)}$ which is the image of the element $1\in \operatorname{Y}(\mathfrak{sl}_{2})$ in the quotient. The highest weight space is one-dimensional.  The Verma module $M(\mu(u))$ is a universal highest weight module in the sense that: if $V(\mu(u))$ is another highest weight module with a highest weight vector $v$, then the mapping $1_{\mu(u)}\mapsto v$ defines a surjective $\operatorname{Y}(\mathfrak{sl}_{2})$-module homomorphism $M\left(\mu(u)\right)\rightarrow V\left(\mu(u)\right)$. The lowest weight modules are defined similarly, and we omit it here.

Every finite-dimensional simple $\operatorname{Y}(\mathfrak{sl}_{2})$-module is a highest weight module \cite{Dr2,YACLA}. Moreover, a simple $\operatorname{Y}(\mathfrak{sl}_{2})$-module $L(\mu(u))$ is finite dimensional if and only if there exists monic polynomial $\pi(u)$ such that $\mu(u)=\frac{\pi(u+1)}{\pi(u)}$, in the sense that the right-hand side is the Laurent expansion of the left-hand side about $u=\infty$.  In the references \cite{ChPr5, ChPr3}, an explicit realization of finite dimensional simple modules for $Y\left(\mathfrak{sl}_2\right)$ is given: every finite dimensional simple module for $Y\left(\mathfrak{sl}_2\right)$ is a tensor product of modules  which are simple under $\mathfrak{sl}_2$. Let $W_m$ be the simple representation of $ \mathfrak{sl}_{2}$ with highest weight $m\in \mathbb{Z}_{\geq 1}$. Pulling it back by evaluation homomorphism $\rho$, $W_m$ becomes to a $\operatorname{Y}(\mathfrak{sl}_{2})$-module. Let $W_m\left(a\right)$ be the simple representation of $\operatorname{Y}(\mathfrak{sl}_{2})$ associated to the Drinfeld polynomial $\pi(u)=(u-a)(u-(a_1+1))\ldots (u-(a+m-1))$.

\begin{lemma}[Proposition 3.5,\cite{ChPr3}]\label{wra}
For any $m\geq 1$, $a\in \mathbb{C}$, $W_m\left(a\right)$ has a basis
$\{w_0,w_1,\ldots, w_m\}$ on which the action of $\operatorname{Y}(\mathfrak{sl}_{2})$ is given by
\begin{center}
$ x_k^+.w_s=\left(s+a\right)^k\left(s+1\right)w_{s+1},\ \ x_k^-.w_s=\left(s+a-1\right)^k\left(m-s+1\right)w_{s-1},$\\

$h_k.w_s=\Big(\left(s+a-1\right)^ks\left(m-s+1\right)-\left(s+a\right)^k\left(s+1\right)\left(m-s\right)\Big)w_s$.
\end{center}
\end{lemma}
The special case $m=1$ will be used in this paper.
\begin{corollary}\label{slw1a}
 In $W_1\left(a\right)$,
 \begin{equation}\label{Cor3.3}
   H_{k}w_1=a^kw_1,\quad X_{k}^{-}.w_1=a^kw_0, \quad X_{k}^{+}.w_0=a^kw_{1},\quad H_{k}.w_0=-a^kw_0.
 \end{equation}
\end{corollary}
\begin{definition}
  A $\operatorname{Y}(\mathfrak{sl}_{2})$-module $V$ is called a weight module if $V=\bigoplus V_{\mu}$, where $\mu\in \mathbb{C}$ and 
  $V_{\mu}=\{v \in V \mid H_{0}.v=\mu.v\}$.
\end{definition}
 The next lemma follows from the defining relations of the Yangian.

\begin{lemma}
 Let $V$ be a weight module and Let $V=\bigoplus V_{\mu}$. If $v\in V_{\mu}$, then $H_i.v\in V_{\mu}$, $X_{i}^{+}.v\in V_{\mu+2}$ and $X_{i}^{-}.v\in V_{\mu-2}$.	
\end{lemma}

We close this section by paraphrasing Theorem 3.4.1 in \cite{Ma} about the dense module for $\mathfrak{sl}_{2}$.

\begin{lemma}\label{lem3} 
Let $V$ be an simple weight $\mathfrak{sl}_{2}$-module which is neither highest nor lowest weight module. Then $V$ is an infinite-dimensional vector space with a basis $\{\ldots, v_{\mu-4}, v_{\mu-2}, v_{\mu}, v_{\mu+2}, v_{\mu+4}, \ldots\}$. The actions $X_{0}^{\pm}$ and $H_{0}$ on $V$ are defined as follows:
$$\begin{aligned} 
 X_0^{-}.\left(v_{\mu+2k}\right) &=v_{\mu+2k-2}, \\ 
 X_0^{+}.\left(v_{\mu+2k}\right) &=a_{\mu+2k} v_{\mu+2k+2},\\
 H_0.\left(v_{\mu+2k}\right) &=(\mu+2k) v_{\mu+2k}. 
\end{aligned}$$
Where $a_{\mu+2k}=\frac{1}{4}\left(\tau-(\mu+2k+1)^{2}\right)\neq 0$ for $\tau\in \mathbb{C}$ and all $k\in \mathbb{Z}$.
\end{lemma}
\section{Dense Modules for $\operatorname{Y}(\mathfrak{sl}_{2})$}

As mentioned in the introduction, a dense module $V$ of $\mathfrak{sl}_{2}$ is a $\operatorname{Y}(\mathfrak{sl}_{2})$-module. By Lemma \ref{lem5}, $\operatorname{Y}(\mathfrak{sl}_{2})$ is generated by $X_{0}^{\pm}$, $H_{0}$ and $H_{1}$. It follows from Lemma \ref{lem3} that the action $H_{1}$ on weight vectors $v_{\mu+2k}$ will totally determine the structure of the $\operatorname{Y}(\mathfrak{sl}_{2})$-module $V$. Suppose that $H_{1}.v_{\mu+2k}=b_{\mu+2k}v_{\mu+2k}$, where $b_{\mu+2k}\in \mathbb{C}$. 
\begin{proposition} Suppose that $\mu\in (0,2]$.
$$b_{\mu+2k} =\frac{\left(\mu+2k\right)\left(a_{u}+b_{\mu}\right)}{\mu}+k\left(\mu+2k\right)-a_{\mu+2k},\qquad k\in \mathbb{Z}$$ 
\end{proposition}
\begin{proof} 
Let $v_{\mu+2k}$ be a weight vector with $\mu\neq 0$.
We first show 
\begin{equation}\label{e18}
  b_{\mu+2k}-2b_{\mu+2k-2}+b_{\mu+2k-4}=6
\end{equation}
for all $k\in \mathbb{Z}$. 
It follows from (\ref{e13}) that \begin{equation*}
\begin{aligned}
X_1^{-}.v_{\mu+2k}&=-\frac{1}{2}\left[H_{1}-\frac{1}{2}H_0^2, X_{0}^{-}\right].v_{\mu+2k}=-\frac{1}{2}\left(b_{\mu+2k-2}-b_{\mu+2k}+2\mu+4k-2\right)v_{\mu+2k-2},\\
X_1^{+}.v_{\mu+2k} &=\frac{1}{2}\left[H_{1}-\frac{1}{2}H_0^2, X_{0}^{+}\right].v_{\mu+2k}=\frac{a_{\mu}}{2}\left(b_{\mu+2}-b_{\mu}-2\mu-4k-2\right)v_{\mu+2k+2}.
\end{aligned}
\end{equation*}
By the defining relation (\ref{e12}), $$\left(\left[X_{1}^{-}, X_{0}^{-}\right]-\left[X_{0}^{-}, X_{1}^{-}\right]\right).v_{\mu+2k}=-\left(X_{0}^{-} X_{0}^{-}+X_{0}^{-} X_{0}^{-}\right).v_{\mu+2k},$$ one can easily obtain that 
\begin{equation*}
\begin{aligned}
  -\frac{1}{2}(b_{\mu+2k-4}-b_{\mu+2k-2}& +2\mu+4k-6).v_{\mu+2k-4}\\
  &+\frac{1}{2}\left(b_{\mu+2k-2}-b_{\mu+2k}+2\mu+4k-2\right).v_{\mu+2k-4}=-v_{\mu+2k-4}
\end{aligned}
\end{equation*}
Comparing the coefficient of $v_{\mu+2k-4}$ we have $b_{\mu+2k}-2b_{\mu+2k-2}+b_{\mu+2k-4}=6$.

We next show that $b_{\mu+2k}=-a_{\mu+2k}+\frac{\mu+2k}{2}(b_{\mu+2k}-b_{\mu+2k-2}-2 \mu-4k+2)$. 
\begin{eqnarray*}
  H_{1} .v_{\mu+2k} &=&\left[X_{1}^{+}, X_{0}^{-}\right] . v_{\mu+2k}, \\
  b_{\mu+2k} v_{\mu+2k} &=&X_{1}^{+} X_{0}^{-} .v_{\mu+2k}-X_{0}^{-} X_{1}^{+}. v_{\mu+2k} ,\\
  b_{\mu+2k} v_{\mu+2k} &=&\frac{a_{\mu+2k-2}}{2}\left(b_{\mu+2k}-b_{\mu+2k-2}-2 \mu-4k+2\right) v_{\mu+2k}\\
  &-&\frac{a_{\mu+2k}}{2}\left(b_{\mu+2k+2}-b_{\mu+2k}-2 \mu-4k-2\right) v_{\mu+2k}.
\end{eqnarray*}
Recall that $a_{\mu+2k-2}=a_{\mu+2k}+\mu+2k$. Thus

\begin{eqnarray*}
 b_{\mu+2k} &=&\frac{a_{\mu+2k}+\mu+2k}{2}\left(b_{\mu+2k}-b_{\mu+2k-2}-2 \mu-4k+2\right)\\
&&-\frac{a_{\mu+2k}}{2}\left(b_{\mu+2k+2}-b_{\mu+2k}-2 \mu-4k-2\right)\\
&=&\frac{a_{\mu+2k}}{2}\left(-b_{\mu+2k+2}+2b_{\mu+2k}-b_{\mu+2k-2}+4\right)\\
&&+\frac{\mu+2k}{2}\left(b_{\mu+2k+2}-b_{\mu+2k}-2 \mu-4k-2\right)\\
 &=&-a_{\mu+2k}+\frac{\mu+2k}{2}(b_{\mu+2k}-b_{\mu+2k-2}-2 \mu-4k+2).
\end{eqnarray*}
Then $(\mu+2k) b_{\mu+2k-2}=(\mu+2k-2)b_{\mu+2k}-2a_{\mu+2k}-2(\mu+2k)^{2}+2(\mu+2k)$.  Adding $(\mu+2k) a_{\mu+2k-2}$ both sides, we get
\begin{equation}\label{e17}
  (\mu+2k)(b_{\mu+2k-2}+a_{\mu+2k-2})=(\mu+2k-2)(b_{\mu+2k}+a_{\mu+2k}-\mu-2k).  
\end{equation}

We proceed the proof by cases according the values of $\mu$.

\textbf{Case 1:} $\mu\in (0,2)$. 

$\mu-2k\neq 0$ for all $k\in \mathbb{Z}$, thus we may rewrite (\ref{e17}) as
$$\frac{b_{\mu+2k-2}+a_{\mu+2k-2}}{\mu+2k-2}=\frac{b_{\mu+2k}+a_{\mu+2k}}{\mu+2k}-1.$$
This recursive relation implies that $\frac{b_{\mu+2k}+a_{\mu+2k}}{\mu+2k}=\frac{b_{\mu}+a_{\mu}}{\mu}+k$ for all $k\in \mathbb{Z}$.
\begin{equation}\label{e14}
 b_{\mu+2 k}=\frac{(\mu+2 k)\left(a_{u}+b_{\mu}\right)}{\mu}+k(\mu+2 k)-a_{\mu+2 k}.
\end{equation}

\textbf{Case 2.} $\mu=2$.

It follows from (\ref{e17}) that $b_{0}=-a_{0}(k=0)$, and then, by (\ref{e18}), $b_{-2}=6-b_{2}-2a_{0}$. A straightforward computation shows that $\frac{b_{-2}+a_{-2}}{-2}=\frac{b_{2}+a_{2}}{2}-2$. Similarly as in Case 1, tedious computations show that $b_{\mu+2k}(k\in \mathbb{Z})$ satisfy (\ref{e14}).
\end{proof} 

\begin{remark}
The dense module for $\operatorname{Y}(\mathfrak{sl}_{2})$ are parametrized by three parameters, $\mu\in (0,2]$, $\tau\in \mathbb{C}$ and $b_{\mu}$ as in (\ref{e14}). Denote it by $V(\mu,\tau,b_{\mu})$.
\end{remark}
\section{Proof of Main Theorem}
Let $V$ be a simple weight module for $\operatorname{Y}(\mathfrak{sl}_{2})$ which admits a one-dimensional weight space. Since the simple finite-dimensional, highest weight, or lowest weight module is a such module, we assume until the proof of Theorem 1.1 that $V$ is none of them above. 

Assume $\dim(V_{\mu})=1$ and $V_{\mu}=\text{span}\{w\}$. Since $V$ is neither highest weight nor lowest weight, we have both $\dim(V_{\mu-2})>0$ and $\dim(V_{\mu+2})>0$.

\begin{lemma}\label{lem1} 
There exists $u\in V_{\mu-2}$ such that $X_{0}^{+}.u=w$. 
\end{lemma}
\begin{proof} 
 Recall that $X_{k}^{+}.V_{\mu-2}\subseteq V_{\mu}$. Suppose on the contrary that for all $u\in V_{\mu-2}$, $X_{0}^{+}.u=0$. We use mathematical induction to show that $X_{n}^{+}.u=0$ for all $n\in \mathbb{N}$. For $n=0$, it is clear. Suppose that $X_{n}^{+}.u=0$ for all $n=0,1,2,\ldots, k$. Let $n=k+1$. 
\begin{equation*}
\begin{aligned}
 X_{k+1}^{+}.u&=\frac{1}{2}[H_0,X_{k+1}^{+}].u\\
       &=\frac{1}{2}\left([H_{1},X_{k}^{+}]-H_{0}X_{k}^{+}-X_{k}^{+}H_{0}\right).u\\
       &=\frac{1}{2}\left(H_{1}X_{k}^{+}-X_{k}^{+}H_{1}-H_{0}X_{k}^{+}-X_{k}^{+}H_{0}\right).u.
\end{aligned}
\end{equation*}
It follows from $H_{i}V_{\mu-2}\subseteq V_{\mu-2}$ and the induction hypothesis that $X_{k+1}^{+}.u=0$. By mathematical induction, we have $x_{n}^{+}u=\mathbf{0}$ for all $n\in \mathbb{N}$. For any $\mathbf{0}\neq u\in V_{\mu-2}$, $W=\operatorname{Y}(\mathfrak{sl}_{2}).u$ is a proper submodule of $V$ since $V_{\mu}\not\subseteq W$, contradicting to the fact that $V$ is simple. Therefore there exists $u\in V_{\mu-2}$ such that $X_{0}^{+}.u=w$. 
\end{proof}

Recall that $V_{\mu}$ is one-dimensional, $H_{k}V_{\mu-2}\subseteq V_{\mu-2}$ and $X_{k}^{+}V_{\mu-2}\subseteq V_{\mu}$. The following corollary is obtained easily. 
\begin{corollary}\label{cor1} 
Let $u$ be as in Lemma \ref{lem1}. For $k,m\in \mathbb{Z}$, $X_{k}^{+}u=a_k w$ and $X_{0}^{+}H_{m} u=b_{m}w$, where $a_k, b_{m}\in \mathbb{C}$.
\end{corollary}

Let $w_1=X_{0}^{+}.w=(X_{0}^{+})^2.u$. Similar as in Lemma \ref{lem1}, we claim that $w_1\neq \mathbf{0}$. Here is a graph of the relations between $u, w, w_{1}$. 
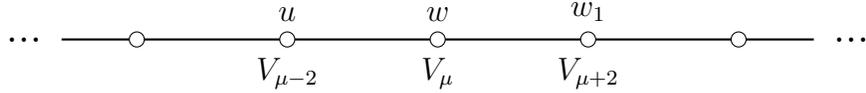
\begin{center}
 \begin{tikzpicture}[scale=0.5]
  \tikzstyle{vertex}=[circle,fill=black!0,minimum size=10pt,inner sep=0pt]
   \coordinate (A) at (-4,0);
   \coordinate (B) at (0,0);
   \coordinate (C) at (4,0);
   \coordinate (D) at (8,0);
   \coordinate (E) at (-8,0);
   \coordinate [label=above:$u$] (1) at ($(A)+(0,0.2)$);
   \coordinate [label=above:$w$] (2) at ($(B)+(0,0.2)$);
   \coordinate [label=above:$w_{1}$] (3) at ($(C)+(0,0.2)$);
    \coordinate [label=below:$V_{\mu-2}$] (4) at ($(A)-(0,0.2)$);
   \coordinate [label=below:$V_{\mu}$] (5) at ($(B)-(0,0.2)$);
   \coordinate [label=below:$V_{\mu+2}$] (6) at ($(C)-(0,0.2)$);
   \draw (A) circle (0.2);
   \draw (B) circle (0.2);
   \draw (C) circle (0.2);
   \draw (D) circle (0.2);
   \draw (E) circle (0.2);
   \draw[thick] ($(E)+(0.2,0)$)--($(A)-(0.2,0)$);
   \draw[thick] ($(A)+(0.2,0)$)--($(B)-(0.2,0)$);
   \draw[thick] ($(B)+(0.2,0)$)--($(C)-(0.2,0)$);
   \draw[thick] ($(C)+(0.2,0)$)--($(D)-(0.2,0)$);
   \draw[thick] ($(-10,0)$)--($(E)-(0.2,0)$);
   \draw[thick] ($(10,0)$)--($(D)+(0.2,0)$);	
   \fill (-11,0)  node [font={\small},scale=1.5, minimum height=3em,vertex] {...};
   \fill (11,0)  node [font={\small},scale=1.5, minimum height=3em,vertex] {...};

 \end{tikzpicture} 
 \end{center}
  
In the next proposition, we will show that $w_{1}$ is a common eigenvector for all $H_{k}$. To prove it, we need the following lemma. 

\begin{lemma}\label{lem2}
$X_{n}^{+}.w=c_{n}w_1$, where $c_{n}\in \mathbb{C}$ and $n\in \mathbb{N}$. 	
\end{lemma}
\begin{proof} 
We prove this lemma by induction. For $n=0$, it is obvious. For $n=1$, 
\begin{equation*}
 \begin{aligned}
   X_{1}^{+}.w&=X_{1}^{+}X_{0}^{+}.u \\
   &= [X_{1}^{+}, X_{0}^{+}].u+X_{0}^{+}X_{1}^{+}.u\\
   &=X_{0}^{+}X_{0}^{+}.u+a_{1}(X_{0}^{+})^2.u\\
   &=c_{1}w_{1}.
 \end{aligned}
 \end{equation*}

Suppose that the statement is true for all $n=0,1,2,\ldots,k$. Let $n=k+1$. 
\begin{equation*}
\begin{aligned}
 X_{k+1}^{+}.w&=X_{k+1}^{+}.(X_{0}^{+}.u) \\
 &= [X_{k+1}^{+}, X_{0}^{+}].u+X_{0}^{+}.(X_{k+1}^{+}.u)\\
 &=(X_{k}^{+}X_{1}^{+}-X_{1}^{+}X_{k}^{+}+X_{k}^{+}X_{0}^{+}+X_{0}^{+}X_{k}^{+}).u+a_{k+1}(X_{0}^{+})^2. u.\\
\end{aligned}
\end{equation*}
It follows from Corollary \ref{cor1} and induction hypothesis that $x_{k+1}^{+}w=c_{k+1}w_{1}$. Therefore, by mathematical induction, the statement is true.
\end{proof}

\begin{proposition}\label{prop1} 
$H_{n}.w_{1}=d_{n}w_{1}$, where $d_{n}\in \mathbb{C}$ and $n\in \mathbb{N}$.
\end{proposition}
\begin{proof} 
We prove this proposition by induction. For $n=0$, $H_{0}.w_{1}=(\mu+2)w_{1}$. Suppose that the statement is true for $n=0,1,2,\ldots,k$. Now let $n=k+1$.  
\begin{equation*}
\begin{aligned}
 H_{k+1}.w_1&=H_{k+1}. (\left(X_{0}^{+}\right)^2.u)\\
 &=[H_{k+1}, \left(X_{0}^{+}\right)^2].u+\left(X_{0}^{+}\right)^2H_{k+1}.u\\
 &= [H_{k+1},X_{0}^{+}]X_{0}^{+}.u+X_{0}^{+}[H_{k+1},X_{0}^{+}].u+X_{0}^{+}\left(X_{0}^{+}H_{k+1}.u\right)\\
 &= \left([H_{k},X_{1}^{+}]+H_{k}X_{0}^{+}+X_{0}^{+}H_{k}\right).(X_{0}^{+}.u)+X_{0}^{+}[H_{k+1},X_{0}^{+}].u+X_{0}^{+}.\left(b_{k+1}w\right)\\
 & \equiv \left([H_{k},X_{1}^{+}]+H_{k}X_{0}^{+}\right).(X_{0}^{+}.u) \quad (\operatorname{mod} \mathbb{C}w_{1})\\
 & \equiv [H_{k},X_{1}^{+}]X_{0}^{+}.u \quad (\operatorname{mod} \mathbb{C}w_{1})\quad (\text{by induction hypothesis})\\
 & \equiv \left(H_{k}X_{1}^{+}-X_{1}^{+}H_{k}\right).w \quad (\operatorname{mod} \mathbb{C}w_{1})\\
 & \equiv c_{1}H_{k}X_{0}^{+}.w \quad (\operatorname{mod} \mathbb{C}w_{1})\\
 &\equiv 0 \quad (\operatorname{mod} \mathbb{C}w_{1}).
\end{aligned}
\end{equation*}
By mathematical induction, we have $H_{n}w_{1}=d_{n}w_{1}$.
\end{proof}

Similar as the discussions from Lemma \ref{lem1} to Proposition \ref{prop1}, we may find a nonzero vector $v\in V_{\mu+2}$ such that $w=X_{0}^{-}v$, and show $w_{-1}=X_{0}^{-}w\neq 0$. Here is a graph of the relations of $v,w$ and $w_{-1}$.
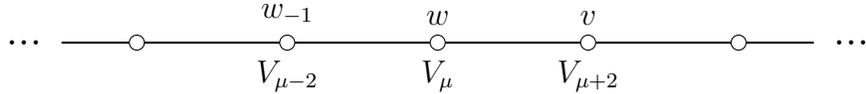
\begin{center}
 \begin{tikzpicture}[scale=0.5]
  \tikzstyle{vertex}=[circle,fill=black!0,minimum size=10pt,inner sep=0pt]
   \coordinate (A) at (-4,0);
   \coordinate (B) at (0,0);
   \coordinate (C) at (4,0);
   \coordinate (D) at (8,0);
   \coordinate (E) at (-8,0);
   \coordinate [label=above:$w_{-1}$] (1) at ($(A)+(0,0.2)$);
   \coordinate [label=above:$w$] (2) at ($(B)+(0,0.2)$);
   \coordinate [label=above:$v$] (3) at ($(C)+(0,0.2)$);
    \coordinate [label=below:$V_{\mu-2}$] (4) at ($(A)-(0,0.2)$);
   \coordinate [label=below:$V_{\mu}$] (5) at ($(B)-(0,0.2)$);
   \coordinate [label=below:$V_{\mu+2}$] (6) at ($(C)-(0,0.2)$);
   \draw (A) circle (0.2);
   \draw (B) circle (0.2);
   \draw (C) circle (0.2);
   \draw (D) circle (0.2);
   \draw (E) circle (0.2);
   \draw[thick] ($(E)+(0.2,0)$)--($(A)-(0.2,0)$);
   \draw[thick] ($(A)+(0.2,0)$)--($(B)-(0.2,0)$);
   \draw[thick] ($(B)+(0.2,0)$)--($(C)-(0.2,0)$);
   \draw[thick] ($(C)+(0.2,0)$)--($(D)-(0.2,0)$);	
   \fill (-11,0)  node [font={\small},scale=1.5, minimum height=3em,vertex] {...};
    \fill (11,0)  node [font={\small},scale=1.5, minimum height=3em,vertex] {...};
    \draw[thick] ($(-10,0)$)--($(E)-(0.2,0)$);
    \draw[thick] ($(10,0)$)--($(D)+(0.2,0)$);
 \end{tikzpicture} 
 \end{center}
  
We will summarize some results regarding the three vectors above, without proof, into the next proposition.
\begin{proposition}\quad
 \begin{enumerate}[label=(\arabic*)]
   \item $X_{n}^{-}w=f_{n}w_{-1}$, where $f_{n}\in \mathbb{C}$ and $n\in \mathbb{N}$. 
   \item $w_{-1}$ is a common eigenvector for all $H_{k}$, where $k\in \mathbb{N}$.
 \end{enumerate} 
\end{proposition}

Define $w_{-n}=(X_{0}^{-})^{n}.w$ and $w_{n}=(X_{0}^{+})^{n}.w$. Denote $w$ by $w_{0}$. Let $$W=\text{span}\{\ldots,w_{-2}, w_{-1}, w_{0}, w_{1}, w_{2}, \ldots\}.$$
Replacing $u$ by $w_{0}$ in the proofs of Lemma \ref{lem2} and Proposition \ref{prop1}, we have the following lemma.

\begin{lemma} 
\begin{enumerate}[label=(\arabic*)]
 \item $X_{n}^{+}.w_{1}=e_{n}w_{2}$, where $e_{n}\in \mathbb{C}$.
 \item $w_{2}$ is a common eigenvector for $H_{k}$.
\end{enumerate}
\end{lemma}
Similarly we have 
\begin{proposition}\label{prop2} 
$w_{n}$ is a common eigenvector for $H_{k}$, where $n\in \mathbb{Z}$ and $k\in \mathbb{N}$.  
\end{proposition}
We now prove our main theorem.
\begin{proof}[Proof of Theorem 1.1]
Let $V$ be a simple weight module for $\operatorname{Y}(\mathfrak{sl}_{2})$ which admits a one-dimensional weight space. It is well known that every simple finite-dimensional module, highest weight module, or lowest weight module has a weight space of one-dimensional. So we may assume $V$ is none of them above. It follows from the discussions above in this section that $W\subseteq V$. 

We claim that $W$ is a $\operatorname{Y}(\mathfrak{sl}_{2})$-module. By Lemma \ref{lem5}, it is enough to show that $W$ is stable under the actions of $X_0^{\pm}$, $H_{0}$ and $H_1$.
\begin{enumerate}[label=(\arabic*)]
  \item $H_0$ stables $W$.
  
  It is easy to show that $H_0.w_{k}=(\mu+2k)w_k$.
  \item $X_{0}^{+}$ stables $W$.
  
  It is enough to show that $X_{0}^{+}.w_{-n}=p_{n}w_{-n+1}$ for all $n\in \mathbb{Z}_{>0}$. We prove it by mathematical introduction. 

  Let $n=1$. 
  \begin{equation*}
  \begin{aligned}
    X_{0}^{+}.w_{-1}&=X_{0}^{+}.(X_{0}^{-}.w_{0})\\
           &=[X_{0}^{+},X_{0}^{-}].w_{0}+X_{0}^{-}.(X_{0}^{+}.w_{0})\\
           &=H_{0}.w_{0}+X_{0}^{-}.w_{1}\\
           &= \mu w_{0}+f_{0}w_{0}\\
           &=(\mu+f_{0})w_{0}.
  \end{aligned}
  \end{equation*}
  Suppose that it is true for $n=1,2,\ldots,k$.
  Let $n=k+1$. 
  \begin{equation*}
    \begin{aligned}
      X_{0}^{+}.w_{-k-1}&=X_{0}^{+}.(X_{0}^{-}.w_{-k})\\
             &=[X_{0}^{+},X_{0}^{-}].w_{-k}+X_{0}^{-}.(X_{0}^{+}.w_{-k})\\
             &=H_{0}.w_{-k}+X_{0}^{-}.(p_{k}w_{-k+1})\\
             &= (\mu-2k)w_{-k}+p_{k}w_{-k}\\
             &= (\mu-2k+p_{k})w_{-k}.
    \end{aligned}
    \end{equation*}
By Mathematical induction, we have $X_{0}^{+}.w_{-n}=p_{n}w_{-n+1}$ for all $n\in \mathbb{Z}_{>0}$. 
  \item $X_0^{-}$ stables $W$.
  å
  Similar to Case (2).
  
  \item $H_1$ stables $W$.
 
  It follows from Proposition \ref{prop2}.
 \end{enumerate}

Since $V$ is simple, $V=W$. Recall that we assume $V$ is an simple weight module for $\operatorname{Y}(\mathfrak{sl}_{2})$ which is neither highest nor lowest weight. Thus $w_{n}\neq 0$ for all $n\in \mathbb{Z}$, so $V$ is infinite-dimensional. It is obvious that $V$ is an irreducible $\mathfrak{sl}_{2}$-module, which is neither highest weight nor lowest weight. It follows from Theorem 3.4.1 in \cite{Ma} that $V$ is a dense module of $\mathfrak{sl}_{2}$. By what we proved in Section 3, $V$ is isomorphic to a dense module $V(\mu,\tau, b_{\mu})$.
\end{proof}


\section{Modules with two-dimensional weight spaces}
For any simple weight $ \mathfrak{sl}_{2}$-module, its weight space is one-dimensional. In this section, we show that this trend does not apply to the Yangian $\operatorname{Y}(\mathfrak{sl}_{2})$, by giving a class of simple weight modules whose every weight space is uniformly two-dimensional.

Let $U=V(\mu,\tau, b_{\mu})\otimes W_{1}(r)$ in this section, where $\tau\neq (\mu+2k+1)^{2}$ for all $k\in \mathbb{Z}$. Immediately we know that $U$ is a weight module for $\operatorname{Y}(\mathfrak{sl}_{2})$, and the weight spaces $U_{\mu+2k+1}=\operatorname{span}\{v_{\mu+2k+2} \otimes w_{-1}, v_{\mu+2k} \otimes w_{1}\}$. 
\begin{lemma}\label{lem4} 
Let $V$ be a nonzero submodule of $U$. If $U_{\mu+2k+1}\subseteq V$ for some $k\in \mathbb{Z}$, then $V=U$.
\end{lemma}
\begin{proof} This lemma will be proved if one can show that both $U_{\mu+2k-1}\subseteq V$ and $U_{\mu+2k+3}\subseteq V$, which follow from the following computations easily. 
 $$
 \begin{aligned}
 X_{0}^{+}.\left(v_{\mu+2k+2}\otimes w_{-1}\right)&=a_{\mu+2k+2} \left(v_{\mu+2k+4} \otimes w_{-1}\right)+\left(v_{\mu+2k+2} \otimes w_{1}\right),\\
 X_{0}^{+}.\left(v_{\mu+2k}\otimes w_{1}\right)&=a_{\mu+2k} \left(v_{\mu+2k+2} \otimes w_{1}\right),\\
 X_{0}^{-}.\left(v_{\mu+2k+2}\otimes w_{-1}\right)&=v_{\mu+2k} \otimes w_{-1},\\
 X_{0}^{-}.\left(v_{\mu+2k}\otimes w_{1}\right)&=v_{\mu+2k} \otimes w_{-1}+v_{\mu+2k-2} \otimes w_{1}.\\
 \end{aligned}
 $$
\end{proof}

\begin{thm} 
 $U$ is simple if and only if $b_{\mu}\neq \frac{\mu^{2}+\mu \pm \mu \sqrt{\tau}}{2}-a_{\mu}+\mu(r-1)$.
\end{thm}
\begin{proof} We first prove the sufficiency by contrapositive method. Suppose on the contrary that $U$ has a proper submodule $V$. As a $\mathfrak{sl}_{2}$-module, $V$ is a weight module, and $V_{\mu+2k+1}=V\cap U_{\mu+2k+1}$, where $k\in \mathbb{Z}$. It follows from Lemma \ref{lem4} that $\dim(V_{\mu+2k+1})\leq 1$. Moreover, there exists $k\in \mathbb{Z}$ such that $\dim(V_{\mu+2k+1})=1$. Without loss of generality, we may assume that $k=0$. 

It follows from the coproduct of $\operatorname{Y}(\mathfrak{sl}_{2})$ that 
\begin{equation*}
\begin{aligned}
 H_{1}.\left(v_{\mu+2} \otimes w_{-1}\right)=&H_{1}.v_{\mu+2} \otimes w_{-1}+H_{0}.v_{\mu+2} \otimes H_{0}.w_{-1}\\
 &+v_{\mu+2} \otimes H_{1}.w_{-1}-2X_{0}^{-}.v_{\mu+2} \otimes X_{0}^{+}.w_{-1}\\
 =& b_{\mu+2}v_{\mu+2} \otimes w_{-1}-(\mu+2)v_{\mu+2} \otimes w_{-1}\\
 &-rv_{\mu+2} \otimes w_{-1}-2v_{\mu}\otimes w_{1}\\
 =& \left(b_{\mu+2}-\mu-2-r\right)\left(v_{\mu+2} \otimes w_{-1}\right)-2\left(v_{\mu} \otimes w_{1}\right).
\end{aligned}
\end{equation*}
Similarly we have 
$$H_{1}.\left(v_{\mu}\otimes w_{1}\right)=\left(b_{\mu}+\mu+r\right)\left(v_{\mu} \otimes w_{1}\right).$$
Let $v=a\left(v_{\mu+2} \otimes w_{-1}\right)+b\left(v_{\mu} \otimes w_{1}\right)$ be a nonzero element in $V_{\mu+1}$. It is not hard to see from Lemma \ref{lem4} that both $a\neq 0$ and $b\neq 0$.
\begin{equation*}\label{e15}
\begin{aligned}
H_{1}.(a&\left(v_{\mu+2} \otimes w_{-1}\right)+b\left(v_{\mu} \otimes w_{1}\right))\\
=&a(H_{1}.(v_{\mu+2} \otimes w_{-1}))+b(H_{1}.(v_{\mu} \otimes w_{1}))\\
=&a(\left(b_{\mu+2}-\mu-2-r\right)\left(v_{\mu+2} \otimes w_{-1}\right)-2\left(v_{\mu} \otimes w_{1}\right))+b((b_{\mu}+\mu+r)(v_{\mu} \otimes w_{1}))\\
=&a(b_{\mu+2}-\mu-2-r)v_{\mu+2} \otimes w_{-1}+(b(b_{\mu}+\mu+r)-2a)v_{\mu} \otimes w_{1}.
\end{aligned}
\end{equation*}
Since $V_{\mu+1}$ is 1-dimensional, $H_{1}.v$ is a scalar of $v$. Then we obtain $$a(b(b_{\mu}+\mu+r)-2a)=ba(b_{\mu+2}-\mu-2-r),$$ which implies that 
$$\frac{a}{b}=\frac{1}{2}(b_{\mu}-b_{\mu+2}+2\mu+4)=r-1-\frac{b_{\mu}+a_{\mu}}{\mu}.$$

A computation shows $X_{0}^{+}.v=aa_{\mu+2}\left(v_{\mu+4} \otimes w_{-1}\right)+\left(a+ba_{\mu}\right)\left(v_{\mu+2} \otimes w_{1}\right)$.
\begin{equation*}\label{e16}
 \begin{aligned}
 H_{1}.(&aa_{\mu+2}\left(v_{\mu+4} \otimes w_{-1}\right)+\left(a+ba_{\mu}\right)\left(v_{\mu+2} \otimes w_{1}\right))\\
 =&aa_{\mu+2}\left(b_{\mu+4}-\mu-4-r\right)\left(v_{\mu+4} \otimes w_{-1}\right)\\
 &+\left(\left(a+ba_{\mu}\right)\left(b_{\mu+2}+2+\mu+r\right)-2aa_{\mu+2}\right)\left(v_{\mu+2} \otimes w_{1}\right)\\
 \end{aligned}
 \end{equation*}
  Since $\dim(V_{\mu+3})\leq 1$, similarly as above, we have 
  $$\frac{a}{b}=\frac{-a_{\mu}\left(2-r+\frac{b_{\mu}+a_{\mu}}{\mu}\right)}{a_{\mu+2}+2-r+\frac{b_{\mu}+a_{\mu}}{\mu}}.$$

 so $\frac{a}{b}=r-1-\frac{b_{\mu}+a_{\mu}}{\mu}=\frac{-a_{\mu}\left(2-r+\frac{b_{\mu}+a_{\mu}}{\mu}\right)}{a_{\mu+2}+2-r+\frac{b_{\mu}+a_{\mu}}{\mu}}$. Let $t=\frac{a}{b}=r-1-\frac{b_{\mu}+a_{\mu}}{\mu}$, then $t=\frac{-a_{\mu}\left(1-t\right)}{a_{\mu+2}+1-t}$. Thus  
$t^{2}-t-a_{\mu+2}t+a_{\mu}t-a_{\mu}=0$. Since $a_{\mu}-a_{\mu+2}=\mu+2$, $t^2+\left(\mu+1\right)t-a_{\mu}=0$. Thus
$t=\frac{a}{b}=\frac{-\left(\mu+1\right)\pm\sqrt{\tau}}{2}$, which implies that $b_{\mu}=\frac{\mu^2+\mu\pm\mu\sqrt{\tau}}{2}-a_{\mu}+\mu(r-1)$. We get a contradiction. Therefore, $U$ is simple.

We next show the necessity. It is equivalent to show that $U$ is reducible if $b_{\mu}=\frac{\mu^2+\mu\pm\mu\sqrt{\tau}}{2}-a_{\mu}+\mu(r-1)$.
Let $v=a\left(v_{\mu+2} \otimes w_{-1}\right)+\left(v_{\mu} \otimes w_{1}\right)$, where $a=\frac{-\left(\mu+1\right)\pm\sqrt{\tau}}{2}$. It follows from above computations that $H_{1}.v=c_{1}v$ and $H_{1}.(X_{0}^{+}.v)=c_{2}X_{0}^{+}.v$, where $c_{1}, c_{2}\in \mathbb{C}$. Similarly one can show that $H_{1}.(X_{0}^{-}.v)=c_{3}X_{0}^{-}.v$ for some $c_{3}\in \mathbb{C}$.  Moreover, $$X_{0}^{-}.(X_{0}^{+}.v)=(a+a a_{\mu+2}+a_{\mu})v_{\mu+2}\otimes w_{-1}+(a+a_{\mu})v_{\mu}\otimes w_{1}.$$ A straightforward computation shows that it is a scalar of $v$.
It follows from the discussions in Section 2 that $V=\operatorname{span}\{\ldots, (X_{0}^{-})^{2}.v, X_{0}^{-}.v,v,X_{0}^{+}.v,(X_{0}^{+})^{2}.v,\ldots\}$ is a submodule whose weight space of weight $\mu+1$ is 1-dimensional. $V$ is a proper submodule of $U$, thus $U$ is reducible.  	
\end{proof}

  


\begin{thebibliography}{2}

\bibitem{VMFY}
Y. Billig, V. Futorny and A. Molev, Verma modules for Yangians, Lett. Math. Phys. 78 (2006), 1-16.

\bibitem{Lau} D. Britten, M. Lau and F. Lemire, {\em Weight modules for current algebras}, J Algebra. {\bf 440} (2015) 245--263.

\bibitem{ChPr5} V. Chari and A. Pressley, \emph{Yangians and R-matrices}, L'Enseign. Math. {\bf 36} (1990), 267--302.

\bibitem{ChPr4} V. Chari and A. Pressley, \emph{Fundamental representations of Yangians and singularities of $R$-matrices}, J.Reine angew. Math. {\bf 417} (1991), 87-128.

\bibitem{ChPr3} V. Chari and A. Pressley, \emph{Yangians: their Representation and Characters}, Acta Applicandae Mathematica, {\bf 44} (1996), 39-58.

\bibitem{Dr2} V. Drinfeld, \emph{A new realization of Yangians and of quantum affine algebras}, (Russian) Dokl. Akad. Nauk SSSR. \textbf{296} (1987), no. 1, 13--17; translation in Soviet Math. Dokl. \textbf{36} (1988), 212--216.

\bibitem{Guay} N. Guay, H. Nakajima and C. Wendlandt, \emph{Coproduct for Yangians of affine Kac-Moody algebras}, Adv Math. \textbf{338} (2018), 865--911.

\bibitem{KhNaPa} S. Khoroshkin, M. Nazarov and P. Papi, {\em Irreducible representations of Yangians}, J Algebra. {\bf 346} (2011) 189--226.


\bibitem{Leve} S. Levendorski\v \i, {\em On generators and defining relations of Yangians}, J. Geom. Phys. {\bf 12} (1993) 1--11.

\bibitem{Ma} V. Mazorchuk, {\em Lectures On $\mathfrak{sl}_2(\mathbb{C})$-modules}, World Scientific Publishing Company, 2009.

\bibitem{Mo2} A. Molev, {\em Irreducibility criterion for tensor products of Yangian evaluation modules}, Duke Math. J. {\bf 112} (2002), 307--341.

\bibitem{YACLA}
A. Molev, \emph{``Yangians and Classical Lie Algebras''}, Mathematical Surveys and Monographs, 143, American Mathematical Society, Providence, RI, 2007.


\bibitem{NT} T. Nakanishi, {\em Fusion, mass, and representation theory of the Yangian algebra}, Nucl. Phys. B. {\bf 439} (1995) 441--460.

\bibitem{TaGu} Y. Tan, and N. Guay. {\em Local Weyl modules and cyclicity of tensor products for Yangians}, Journal of Algebra {\bf 432} (2015): 228-251.


\end{thebibliography}
\end{document}